\documentclass[11pt]{amsart}
\usepackage{amsmath, amssymb, verbatim,color, amscd}
\usepackage[mathscr]{eucal}

\newtheorem{theorem}{Theorem}[section]
\newtheorem{lemma}[theorem]{Lemma}

\renewcommand{\geq}{\geqslant}

\theoremstyle{definition}

\newtheorem{example}[theorem]{Example}

\theoremstyle{definition}
\newtheorem{remark}[theorem]{Remark}
\numberwithin{equation}{section}

\newcommand{\ZZ} {\mathbb{Z}}

\newcommand{\HH} {\mathbb{H}}
\newcommand{\RR} {\mathbb{R}}
\newcommand{\CC} {\mathbb{C}}

\newcommand{\PP} {\mathbb{P}}

\renewcommand{\geq}{\geqslant}

\newcommand{\Hom}{\mathrm{Hom}}
\newcommand{\fT}{\mathfrak{T}}
\newcommand{\fS}{\mathfrak{S}}
\newcommand{\cP}{\mathcal{P}}

\newcommand{\FS}{\mathrm{FS}}

\def\vth{\vartheta}

\def\cX{\mathcal{X}}

\def\sO{\mathscr{O}}
\def\sL{\mathscr{L}}

\def\la{\langle}
\def\ra{\rangle}

\numberwithin{equation}{section} \numberwithin{figure}{section}

\begin{document}
\title[Balanced embedding of degenerating  Abelian varieties]{Balanced embedding of degenerating  Abelian varieties}
\author[X. Wang]{Xiaowei Wang}
\thanks{ }
\address{Department of Mathematics and Computer Sciences, Rutgers University, Newark NJ 07102-1222, USA}
\email{xiaowwan@rutgers.edu}
  \author[Y. Zhang]{Yuguang Zhang}
 \thanks{ The first named  author is partially supported by a Collaboration Grants for Mathematicians from Simons Foundation and  the second named  author is supported in part by  grant NSFC-11271015. }
\address{Yau Mathematical Sciences Center,  Tsinghua University,  Beijing 100084, P.R.China.}
\email{yuguangzhang76@yahoo.com}

\begin{abstract}
For a certain maximal unipotent  family of    Abelian varieties over  the punctured disc, we show that
after a  base  change, one can complete the  family  over a disc  such that  the whole degeneration  can be simultaneously balanced embedded into  a projective space  by the    theta functions.  Then we study the   relationship   between  the balanced filling-in and the Gromov-Hausdorff limit of flat K\"{a}hler metrics on the nearby fibers.
\end{abstract}

\maketitle
\section{Introduction}
Let $X\subset\CC\PP^{\nu}$ be a  $n$-dimensional  projective variety. The embedding $X \hookrightarrow \mathbb{CP}^{\nu}$ is called balanced   if $$\int_{X} \big(\frac{z_{i}\bar{z}_{j}}{|z|^{2}}-\frac{\delta_{ij}}{\nu+1}\big)\omega_{FS}^{n}=0,  $$ where $z_{0}, \cdots, z_{\nu}$ are  the homogenous coordinates of  $\mathbb{CP}^{\nu}$, $|z|^{2}=\sum|z_{i}|^{2}$,  and $\omega_{FS}$ is the Fubini-Study metric. In  \cite{ZS},  S. Zhang  proved  that $X$,   as an algebraic cycle in $\CC\PP^\nu$,   is   Chow polystable  if and only if  the embedding  $X\subset \CC\PP^{\nu}$ can be translated to a balanced one via an  element $u\in \mathrm{SL}(\nu+1)$.

A theorem due to Donaldson \cite{Dos} shows the connection of  the balanced embedding and  the existence of K\"{a}hler metric with constant scalar curvature. More precisely, let $(X,L)$ be a polarized manifold of dimension $n$ such that the  automorphism group ${\rm Aut}(X,L)$ is finite. If there is a K\"{a}hler metric $\omega$ with constant scalar curvature representing $c_{1}(L)$, then Donaldson's theorem asserts that for $k\gg 1$, $L^{k}$ induces a balanced embedding $\Phi_{k}: X\hookrightarrow \mathbb{CP}^{\nu_{k}}$ with $\Phi_{k}^{*}\mathcal{O}_{\mathbb{CP}^{\nu_{k}}}(1)=L^{k}$. Furthermore, $$\|\omega-k^{-1}\Phi_{k}^{*}\omega_{FS}\|_{C^{r}(X)}\rightarrow 0,$$ when $k \rightarrow\infty$,   in the $C^{r}$-sense for any $r>0$.  The K\"{a}hler metric $k^{-1}\Phi_{k}^{*}\omega_{FS}$ is called a {\em balanced matric.}

  If $(X,L)$ is a polarized Calabi-Yau manifold, Yau's theorem on  the Calabi conjecture says  that there exists a unique Ricci-flat K\"{a}hler-Einstein metric $\omega$ with $\omega\in c_{1}(L)$, i.e. the Ricci curvature ${\rm Ric}(\omega)\equiv 0$  (cf. \cite{Yau1}). By Donaldson's theorem, $L^{k}$ induces a balanced embedding  $X\hookrightarrow \mathbb{CP}^{\nu_{k}}$ for $k\gg 1$, and the  Ricci-flat K\"{a}hler-Einstein metric $\omega$ can be approximated by the  balanced metrics. If $(X,L)$ is a  principally polarized  Abelian variety, it is proven first in \cite{No}  (and also independently \cite{WY}) that the standard embedding induced by the  classical theta functions of level $k$ is balanced.  Moreover, in this case the convergence of balanced metrics to the flat metric can be verified  via a  complete elementary way without quoting \cite{Dos}.

In  \cite{SYZ}, Strominger, Yau and Zaslow propose   a geometric way of  constructing mirror Calabi-Yau manifolds via dual special lagrangian fibration,  which is the celebrated  SYZ conjecture.  Later, a new version of the  SYZ conjecture is proposed by   Gross,  Wilson,  Kontsevich, and  Soibelman,
    (cf. \cite{GW,KS,KS2}) by using   the collapsing of Ricci-flat K\"{a}hler-Einstein metrics.
  Let   $(\mathcal{X} \rightarrow \Delta, \mathcal{L}) $ be   a maximal unipotent  degeneration of polarized   Calabi-Yau $n$-manifolds, i.e. the relative canonical bundle $\mathcal{K}_{\mathcal{X} / \Delta} $ is trivial, such that  $0\in \Delta$ is  a large complex limit  point, and $\omega_{t}$ be the  Ricci-flat K\"{a}hler-Einstein metric  satisfying  $\omega_{t}\in c_{1}(\mathcal{L}|_{X_{t}}) $  for $t\in \Delta^{\circ}$.
  The collapsing   version of SYZ conjecture  asserts that
       $$(X_{t}, {\rm diam}_{\omega_{t}}^{-2}(X_{t}) \omega_{t})\rightarrow (B,d_{B})$$  in the Gromov-Hausdorff sense, when $t\rightarrow 0$, where $(B,d_{B})$ is a compact metric space.
 Furthermore,  there is an open dense subset $B_{0}$ of $ B$, which is smooth, and
   is  of real dimension $n$,  and   admits    a real  affine structure. The metric $d_{B}$ is induced by a Monge-Amp\`ere  metric $g_{B}$ on $B_{0}$, i.e.  under  affine coordinates $y_{1},  \cdots, y_{n}$,  there is a potential function $\varphi$ such that $$g_{B}= \sum_{ij} \frac{ \partial^{2} \varphi}{ \partial y_{i}  \partial y_{j}} dy_{i} dy_{j},    \   \  {\rm and}  \  \   \det \Big(\frac{\partial^{2}\varphi}{ \partial y_{i}  \partial y_{j} }\Big ) ={\rm const.}.$$  Clearly  it  is true for Abelian varieties.
This conjecture is verified by Gross and Wilson for fibred  K3 surfaces with only type $I_{1}$ singular fibers  in \cite{GW}, and is studied for higher dimensional HyperK\"ahler manifolds in \cite{GTZ, GTZ2}.

 Bernd Siebert raises  a question   to  relate the balanced embeddings of $X_{t}$ to the   metric limit of rescaled  Ricci-flat K\"{a}hler-Einstein metrics $\varepsilon_{t}\omega_{t}$ for a  certain family of constants $\varepsilon_{t}$. In the Gross-Siebert program (cf. \cite{Gro,GS1}),  theta functions are constructed  on certain     degenerations of polarized  Calabi-Yau manifolds $(\mathcal{X}\rightarrow {\rm Spec}\mathbb{C}[[t]],\mathcal{L})  $  as the canonical basis of the space of sections for $\mathcal{L}$ (cf. \cite{GHKS,GS,GHKS2}), which is predicted by the Homological Mirror Symmetry conjecture.  In particular,  these theta functions recover  the classical theta functions in the case of principally polarized  Abelian varieties. If $(\mathcal{X},\mathcal{L})$ is an analytic family (cf. \cite{RS}),  Siebert asks whether  the  theta functions give the   balanced embeddings of polarized  Calabi-Yau manifolds, and furthermore whether  there is  a family version of the Donaldson's theorem for the degeneration of Calabi-Yau manifolds near large complex limits.
 In this note, we study this question in the case of  principally polarized   Abelian varieties, and establish  the connection between the limit metric $g_{B}$ and the balanced embeddings.

In Section 2,  let us recall the  basic setup for  a maximal unipotent  family of   principally polarized  Abelian varieties over the punctured  disc, and then state our  main result.  Theorem \ref{thm01} says that after a certain base  change, one can find a filling-in to complete the  family of Abelian varieties to a   degeneration,   such that  the whole   degeneration  can be simultaneously balanced embedded in  a projective space over a disc by the {\em canonical}  theta functions constructed via Gross-Siebert program.   Theorem \ref{thm02} studies the relationship    between  the balanced filling-in and the Gromov-Hausdorff limit of flat K\"{a}hler metrics on the nearby fibers.  In Section 3, we review  the construction of theta functions on degenerations of Abelian varieties  in the Gross-Siebert program.  Finally,  Theorem \ref{thm01} and Theorem \ref{thm02} are proved in Section 4.

\noindent {\bf Acknowledgements:}  The authors  would  like to thank Helge Ruddat, Mark Gross and Bernd Siebert for many helpful discussions.  The work was done when both authors are  in residence at the Mathematical Sciences Research Institute in Berkeley, California, during the 2016 Spring semester.

\section{Set up and Main Theorems}
In this paper, we always denote
   $M\cong \mathbb{Z}^{n}$,   $M_{\mathbb{R}}=M\otimes_{\mathbb{Z}}\mathbb{R}$,  $N=\Hom_{ \mathbb{Z}}(M, \mathbb{Z})$,   $N_{\mathbb{R}}=N\otimes_{\mathbb{Z}}\mathbb{R}$, $T_{\mathbb{C}}=N\otimes_{\mathbb{Z}}\mathbb{C}^{*}$, and $\langle\cdot,\cdot\rangle$ the pairing between $N_{\mathbb{R}}$ and $M_{\mathbb{R}}$.

\subsection{A family of Abelian varieties}\label{ppav}
This subsection gives the basic setup of this paper, which is  a  family of Abelian varieties over  the punctured disc approaching to  a large complex limit.

   Let $Z(\cdot, \cdot): M_{\mathbb{R}} \times M_{\mathbb{R}} \rightarrow \mathbb{R}$  be a  
   positive  definite   bilinear  form satisfying  $Z(M,M)\subset \mathbb{Z}$.  If  we define the quadratic  function
 \begin{equation}\label{z-phi}
 \overline{\varphi}(y)=\frac{1}{2}Z(y,y)
 \end{equation}
 on $M_{\mathbb{R}}$, and the affine  linear function
     \begin{equation}\label{e0.0}\alpha_{\gamma}(\cdot)= Z(\gamma, \cdot)+\frac{1}{2} Z(\gamma, \gamma),
     \end{equation}
 for any $\gamma \in M$,    then  \begin{equation}\label{e0.1} \overline{\varphi}(y+\gamma)= \overline{\varphi} (y) + \alpha_{\gamma} (y).\end{equation}
      The couple $\{M, Z\}$  determines  a family of principally polarized Abelian varieties over the punctured  disc  $(\pi: \mathcal{X}_{\eta}\rightarrow \Delta^{\circ}, \mathcal{L}_{\eta})$ as the following.

  We define an $M$-action  on  $M\times \mathbb{Z}$ via $(m,r) \mapsto (m,r+  Z(\gamma,m))$ for any $\gamma\in M$, which induces an $M$-action on $ T_{\mathbb{C}} $   by
        \begin{equation}\label{e0.100}\mathcal{Z}^{m} \mapsto \mathcal{Z}^{m} s^{Z(\gamma, m)},  \  \  \   \gamma\in M,   \end{equation} for any $s\in \Delta^{\circ}$.  More explicitly,
let
      $$e_{1}, \cdots, e_{n}\in M$$
       be a  basis of $M$,  we have coordinates  $z_{1}= \mathcal{Z}^{e_{1}}, \cdots, z_{n}=\mathcal{Z}^{e_{n}}$ on $T_{\mathbb{C}}\cong (\mathbb{C}^\times)^{n}$, and   the $M$-action    is  that $ (z_{1}, \cdots, z_{n}) \mapsto (z_{1} s^{Z(\gamma, e_{1})},  \cdots, z_{n} s^{Z(\gamma, e_{n})}) $  for any $\gamma\in M$.     We claim that the quotient $X_{s}=T_{\mathbb{C}}/M$ is a principally  polarized Abelian variety with period matrix $$\left[ I,   \frac{\log s}{2\pi \sqrt{-1}}Z_{ij} \right]$$ where $Z_{ij}=Z(e_{i},e_{j}) \in \mathbb{Z} $.

      Denote $e_{1}^{*}, \cdots, e_{n}^{*}$ the dual basis of $N$, and  $N_{\mathbb{C}}=N\times_{\mathbb{Z}}\mathbb{C} \cong \mathbb{C}^{n}$. We have a natural embedding $N\hookrightarrow  N_{\mathbb{C}} $ as the real part, and    by abusing notions,  we regard $e_{1}^{*}, \cdots, e_{n}^{*}$ as a $\mathbb{C}$-basis of $N_{\mathbb{C}}$.
    The universal covering  $N_{\mathbb{C}} \rightarrow T_{\mathbb{C}} $ is given by $ w_{i} \mapsto z_{i}= \exp 2\pi \sqrt{-1} w_{i}  $, $i=1, \cdots, n$, where $w_{1}, \cdots, w_{n}$ are coordinates of $N_{\mathbb{C}}$ respecting to $e_{1}^{*}, \cdots, e_{n}^{*}$.   Then
        $T_{\mathbb{C}}\cong N_{\mathbb{C}}/N$ where $N$ acts on $N_{\mathbb{C}}$ given  by $ w_{i} \mapsto  w_{i} + \langle \mu, e_{i}\rangle=  w_{i}+ \mu_{i}$ for any $\mu=\sum\limits_{i}\mu_{i}e_{i}^{*}\in N$. We have   an $N \times M$-action on     $\mathbb{C}^{n}$  by $$w_{i} \mapsto w_{i} +\langle \mu, e_{i}\rangle +  \frac{\log s}{2\pi \sqrt{-1}}Z(\gamma, e_{i}), $$ for any  $(\mu, \gamma)\in N\times M$, and we obtain  $X_{s}=T_{\mathbb{C}}/M=\mathbb{C}^{n}/\Lambda_{s}$, where the lattice $\Lambda_{s}={\rm span}_{\mathbb{Z}} \{e_{1}^{*}, \cdots, e_{n}^{*},    \frac{\log s}{2\pi \sqrt{-1}}Z( e_{1}, \cdot)   , \cdots,    \frac{\log s}{2\pi \sqrt{-1}}Z( e_{n}, \cdot)\}$.     Furthermore,  we construct  a family of Abelian varieties $\mathcal{X}_{\eta}=(T_{\mathbb{C}} \times \Delta^{\circ} )/M\rightarrow \Delta^{\circ}$ over the punctured disc $\Delta^{\circ}$ with fiber $X_{s}$.

        We extend the $N \times M$-action on     $\mathbb{C}^{n}$ to  $\mathbb{C}^{n}\times \mathbb{C}$ by $$(w, \lambda)  \mapsto (w + \mu +  \frac{\log s}{2\pi \sqrt{-1}}Z(\gamma, \cdot), \lambda \exp \pi\sqrt{-1}(- \frac{\log s}{2\pi \sqrt{-1}}Z(\gamma,\gamma)-2 \langle w, \gamma \rangle)),  $$ for any  $(\mu, \gamma)\in N\times M$, where $w=(w_{1}, \cdots, w_{n})\in \mathbb{C}^{n}$. The quotient $(\mathbb{C}^{n}\times \mathbb{C})/(N \times M)$ is the relative ample bundle $\mathcal{L}_{\eta}$.  The classical Riemann theta function (cf. \cite{BL})
$$
 \vartheta=\sum_{\gamma\in M}\exp \pi \sqrt{-1}\left( \frac{\log s}{2\pi \sqrt{-1}}Z(\gamma, \gamma)+ 2 \langle w,\gamma \rangle\right)$$
 is the distinguished section of $\mathcal{L}_{\eta}$.
On any $X_{s}$, the first Chern class $c_{1}(\mathcal{L}_{\eta}|_{X_{s}})$ is represented by the flat K\"{a}hler metric \begin{equation}\label{e0.10}\omega_{s}  =\frac{-\pi \sqrt{-1}}{ \log |s|}\sum_{ij} Z ^{ij} dw_{i}\wedge d\bar w_{j},    \end{equation} where $ Z ^{ij}$ denotes the inverse of the matrix $Z_{ij}=Z(e_{i},e_{j}) $.

 Let $x_{1}, \cdots, x_{n}, y_{1}, \cdots, y_{n}$ denote the  coordinates on $N_{\mathbb{C}}$  with respect to  the basis $e_{i}^{*}$, $1\leqslant i\leqslant n$,    $ (  \frac{\log |s|}{2\pi \sqrt{-1}}+ \arg (s))Z( e_{j}, \cdot)   $,   $1\leqslant j\leqslant n$, where    $0\leqslant \arg(s) <2\pi$.  We also  regard   $ y_{1}, \cdots, y_{n}$ (resp.  $x_{1}, \cdots, x_{n}$)    as coordinates on $M_{\mathbb{R}} $    (resp.  $N_{\mathbb{R}} $) respecting to $e_{1}, \cdots, e_{n}$.
Then
$$w_{i}=x_{i}+(\arg (s)-\sqrt{-1}\frac{\log |s|}{2\pi }  )\sum_{j=1}^nZ_{ij}y_{j},\  i=1, \cdots, n,$$ and
 as a symplectic form,   $$\omega_{s}=\sum\limits_{ij}dx_{i} \wedge dy_{j}.   $$ The  corresponding  Riemannian metric is  $$g_{s}=\frac{-2\pi}{\log |s|}\sum_{i,j} Z^{ij}dx_{i}dx_{j}-\frac{\log |s|}{2\pi} \sum_{i,j} Z_{ij}dy_{i}dy_{j},   $$ and,   when $s  \rightarrow 0$,  $$(X_{s},  \frac{-2\pi}{\log |s|}g_{s})     \longrightarrow (B, g_{B}=    \sum_{ij} Z_{ij}dy_{i}dy_{j})$$ in the Gromov-Hausdorff topology,  where  $B=M_{\mathbb{R}}/M$.

 We can  regard $\overline{\varphi}$ as a multivalue  function on $B=M_{\mathbb{R}}/M $ by (\ref{e0.1}),  and on any small open subset on $B$,  the difference of any two sheets of $\overline{\varphi}$ is a linear function defined in (\ref{e0.0}).  Note that the Hessian matrix of $\overline{\varphi}$ is well-defined, and $\frac{\partial^{2} \overline{\varphi}}{\partial y_{i} \partial y_{j}}=Z_{ij}$.
    The Riemannian metric $g_{B}$ is the Monge-Amp\`{e}re metric with potential $\overline{\varphi}$ respecting to the affine coordinates $y_{1}, \cdots, y_{n}$, i.e. \begin{equation}\label{e0.11}g_{B}=    \sum_{ij} \frac{\partial^{2} \overline{\varphi}}{\partial y_{i} \partial y_{j}}dy_{i}dy_{j}, \  \ {\rm and} \  \  \det ( \frac{\partial^{2} \overline{\varphi}}{\partial y_{i} \partial y_{j}}) \equiv  \det(Z_{ij}).  \end{equation}

\subsection{Main results}
 For any $k\gg 1$, let $M_{k}=kM$,   and  $\varphi$ be an   $M_{k}$-periodic convex piecewise linear function such that the  slopes   of $\varphi$ are in $N$, and   \begin{equation}\label{e0.2}\varphi(y+\gamma)= \varphi (y) + \alpha_{\gamma} (y) \end{equation}  for any $\gamma\in kM$,  which induces  an  $M_{k}$-invariant rational  polyhedral decomposition $\mathcal{P} $ of $M_{\mathbb{R}}$, i.e.
   $\sigma $ is a cell of $\mathcal{P} $ if and only if $\varphi$ is linear on $\sigma$.

 The Mumford's construction (cf. \cite{Mum}) gives a  degeneration of principally polarized   Abelian varieties  $(\pi: \mathcal{X}\rightarrow \Delta, \mathcal{L})$ from the data $(M_{k},\mathcal{P}, \varphi) $ such that  $ \pi: \mathcal{X}_{\Delta^{\circ}}\rightarrow \Delta^{\circ}$ is the base  change of $ \mathcal{X}_{\eta}$ via $t\mapsto t^{k}=s$, where $ \mathcal{X}_{\Delta^{\circ}}=\pi^{-1}(\Delta^{\circ})$,  and $\mathcal{L}|_{\mathcal{X}_{\Delta^{\circ}}}$ is the pull-back of  $\mathcal{L}_{\eta}^{k}$. The central fiber $X_{0}=\pi^{-1}(0)$ is reduced and reducible with only  toric singularities.   The intersection complex of  $X_{0}$ is $(B_{k}, \tilde{\mathcal{P}})$ where $B_{k}=M_{\mathbb{R}}/M_{k}$ and $\tilde{\mathcal{P}}$ is the quotient  rational polyhedron decomposition  of   $\mathcal{P}$.  The irreducible components are one to one corresponding to $n$-cells in $\tilde{\mathcal{P}}$, and for any $n$-cells  $\sigma\in \tilde{\mathcal{P}}$, the respective irreducible component is the toric variety $X_{\sigma}$ defined by $\sigma$, where we regard $\sigma$ as a polytope in $M_{\mathbb{R}}$. Furthermore,  the restriction of  $\mathcal{L}$ on $X_{\sigma}$ is the toric ample line bundle  defined by $\sigma$.

   For any $m\in B_{k}(\mathbb{Z})=M/M_{k}$,  a section $\vartheta_{m}$ of $\mathcal{L}$ is constructed in Section 6 of  \cite{GHKS} such that the restriction of $ t^{\varphi(m)-\bar{\varphi} (m)}\vartheta_{m}$ on any $X_{t}=\pi^{-1}(t)$, $t\neq0$, is a classical Riemann  theta function (See also \cite{GS}). And the restriction of  $\vartheta_{m}$ on any component $X_{\sigma}$ of $X_{0}$ is a monomial section of $\mathcal{L}|_{X_{\sigma}} $.

 Note that the choice of $(\mathcal{P}, \varphi)$  is not unique, and different choices  give  different filling-ins  $X_{0}$. However there is a canonical one studied in \cite{AN},  which  satisfies $\varphi (m)=\bar{\varphi} (m)$ for any $m\in M$. More precisely, let $\varphi: M_{\mathbb{R}}   \rightarrow \mathbb{R}$ be the convex   piecewise linear function such that     the graph of $\varphi$ is the lower bound of the convex hull of $     \{(m, \overline{\varphi}(m))| m\in M \}\subset M_{\mathbb{R}}\times \mathbb{R}$, and let  $\mathcal{P}$ be the rational  polyhedral decomposition of $ M_{\mathbb{R}}$ induced by $\varphi$, i.e. a cell $\sigma\in\mathcal{P} $ if and only if $\varphi$ is linear on $\sigma$. It is clear that (\ref{e0.2}) is satisfied, and   \begin{equation}\label{enew01} \varphi(m)= \overline{\varphi}(m),  \  \  \  \  {\rm for \ \ any} \ m\in M.  \end{equation}
Let $\overline{\cP}$ be the boundary  of the convex hull of the lattice points on the graph of $\overline{\varphi}$, then the polyhedral decomposition $\cP$ is obtained in such a way that each cell of $\cP$ is precisely the projection of a face of $\overline{\cP}$ onto $M_\RR$.
The decomposition $\mathcal{P}$ is the  {\em mostly divided}  polyhedral decomposition  that one can have.  Any cell $\sigma$ of $\mathcal{P}$ intersects with the lattice $M$ only  at its vertices, i.e. there is no integral point in the interior of $\sigma$.

 We further assume that for any $\sigma\in\mathcal{P}$,  the slop of $\varphi|_{\sigma} $ is integral, i.e.
 \begin{equation}\label{e0.3}
d\varphi|_{\sigma} \in N \end{equation}
   for simplicity. It is not a further  restriction, and the reason is as the following. Note that $d\varphi|_{\sigma} \in N\otimes_{\mathbb{Z}}\mathbb{Q}$. For any $m\in M$,  there is a   $\nu \in \mathbb{N}$   such that  $\nu d\varphi|_{\sigma} \in N$ for any cell $\sigma\in\mathcal{P}$ with $m\in \sigma$, i.e. $\sigma$ belongs to  the star of $m$.  By the $M$-action and (\ref{e0.2}),   $\nu d\varphi|_{\sigma}+ \nu d\alpha_{\gamma} \in N$ for any $\gamma\in M$, and thus $\nu \varphi$ satisfies (\ref{e0.3}).
   If we replace $Z $ by $\nu Z $, then
     the  new family of Abelian varieties constructed from $\nu\overline{\varphi}$ is the  base change of the original family by $s \mapsto s^{\nu}$.
   Hence we assume (\ref{e0.3}).

  Before we present  the main theorems,  we look at some lower dimensional cases.
  If  $\dim_{\mathbb{R}}M_{\mathbb{R}}=1$,  $B_{k}$ is a cycle, $1$-cells in $\mathcal{P} $ are intervals, and  $X_{0}$ is the Kodaira type I$_{k}$ fiber, i.e.  a cycle of $k$ rational curves. It is well known to experts that
  $X_{0}$ can be balanced embedded.
   When $\dim_{\mathbb{R}}M_{\mathbb{R}}=2$,    there are only two  possible choices of $\mathcal{P} $ by \cite{AN}. One is that any $2$-cell in $\mathcal{P} $ is a standard simplex, and the other one is that  $\mathcal{P} $ consists of  cubes.   Hence $X_{0}$ consists either finite many $\mathbb{CP}^{2}$ or finite many $\mathbb{CP}^{1}\times \mathbb{CP}^{1}$.

 \begin{center}
  \setlength{\unitlength}{0.5cm}
\begin{picture}(3.3,2.6)(-2.5,-0.25)
\linethickness{.075mm}
\put(-2,0.5){\line(0,1){1}}
\put(-2,1.5){\line(1,1){1}}
\put(-2,0.5){\line(1,0){1}}
\put(-1,0.5){\line(1,1){1}}
\put(-1,2.5){\line(1,0){1}}
\put(0,1.5){\line(0,1){1}}
\end{picture}
   \setlength{\unitlength}{0.5cm}
\begin{picture}(4.3,3.6)(-2.5,-0.25)
\linethickness{.075mm}
\multiput(-2,0)(1,0){4}
{\line(0,1){3}}
\multiput(-2,0)(0,1){4}
{\line(1,0){3}}
\end{picture}
\setlength{\unitlength}{0.5cm}
\begin{picture}(4.3,3.6)(-2.5,-0.25)
\linethickness{.075mm}
\multiput(-2,0)(1,0){4}
{\line(0,1){3}}
\multiput(-2,0)(0,1){4}
{\line(1,0){3}}
\put(0,0){\line(1,1){1}}
\put(-1,0){\line(1,1){2}}
\put(-2,0){\line(1,1){3}}
\put(-2,1){\line(1,1){2}}
\put(-2,2){\line(1,1){1}}
\end{picture}\end{center}

   The first result of this paper shows that the embedding of  the canonical degeneration $\mathcal{X}$ of \cite{AN} by  theta functions constructed in \cite{GHKS} is balanced.

 \begin{theorem}\label{thm01}  For any $k\gg 1$,  $B_{k}=M_{\mathbb{R}}/M_{k}$,  $B_{k}(\mathbb{Z})=M/M_{k}$, and let $(\pi: \mathcal{X}\rightarrow \Delta, \mathcal{L})$ be the  degeneration of principally polarized   Abelian varieties  from the triple  $(M_{k},\mathcal{P}, \varphi) $.  The theta functions $\{\vartheta_{m}| m\in B_{k}(\mathbb{Z})\}$ define a relative balanced embedding $$ \Phi_{ k} =[\vartheta_{m_{1}}, \cdots, \vartheta_{m_{k^{n}}}]: \mathcal{X}\rightarrow \mathbb{CP}^{k^{n}-1}\times \Delta,$$ i.e. for any $t\in \Delta$, $\Phi_{ k}|_{X_{t}}:X_{t}\hookrightarrow \mathbb{CP}^{k^{n}-1}$ is a balanced embedding.
  \end{theorem}

Again it was proven in \cite{WY} and \cite{No} that  any individual Abelian variety  $X_{t}$, $t\neq0$, can be balanced embedded in certain $ \mathbb{CP}^{N}$ via theta functions.  If one set  $t\rightarrow 0$, the limit variety $X_{0}$ in the projective space is also balanced. However Theorem \ref{thm01} uses a group of different theta functions that guarantee the balanced embedding varying  holomorphicly    when $t$ approaches to $0$. Furthermore, Theorem \ref{thm01} identifies  the balanced limit  $X_{0}$ to be  the canonical filling-in of $ ( \mathcal{X}_{\Delta^{\circ}}, \mathcal{L}|_{\mathcal{X}_{\Delta^{\circ}}})$ constructed in \cite{AN}.

Now we study the connection   between the metric limit  $(B, g_{B})$ and the balanced filling-in when $k\rightarrow \infty$.  For a fixed $k\gg 1$, $B_{k}$ has a natural affine structure induced by  $M_{\mathbb{R}}$.  If $f$ is a convex function on an open subset $U$ of $B_{k}$ with  respect  to the affine structure, where  we also regard $U$ as a subset of $M_{\mathbb{R}}$ by the quotient, then for any $y_{0}\in U$, we let $$ \partial f (y_{0}) =\{\upsilon \in N_{\mathbb{R}}| f(y)\geq \langle \upsilon, y-y_{0} \rangle + f(y_{0}) \  \ {\rm  for \ \  all}  \  \ y\in M_{\mathbb{R}} \},  $$  and we define the  Monge-Amp\`{e}re measure $${\rm MA}(f)(E)={\rm Vol}(\bigcup_{y\in E} \partial f (y)),$$ for any Borel subset $E\subset U$, where ${\rm Vol}$ denotes the standard Euclidean measure on  $N_{\mathbb{R}}$ (cf. \cite{BPS} and \cite{TW}).  It is well known  that ${\rm MA}(f+\alpha)={\rm MA}(f)$ for any linear function $\alpha$ on $M_{\mathbb{R}}$, and if $f$ is smooth, $${\rm MA}(f)=\det \big(\frac{\partial^{2}f}{\partial y_{i} \partial y_{j}}\big)dy_{1}\wedge \cdots \wedge dy_{n}. $$

Note that we can regard the function $\varphi$ as a multivalue function on $B_{k}$, and by (\ref{e0.2}), the difference between  any  two sheets of $\varphi$ is a linear function. Thus we have a well-defined Monge-Amp\`{e}re measure  ${\rm MA}(\varphi)$.   The next theorem shows that  after some rescaling, this Monge-Amp\`{e}re measure converges to  the Monge-Amp\`{e}re measure of the potential function $\bar{\varphi}$ of   $g_{B}$  when $k\rightarrow \infty$, and furthermore, the rescaled  potential function $\varphi$  also converges to $\bar{\varphi}$ in the $C^{0}$-sense, which shows the link between the balanced embeddings and the Gromov-Hausdorff limit.

\begin{theorem}\label{thm02} For any $k \gg 1$, the Monge-Amp\`{e}re measure of $\varphi$ is $${\rm MA}(\varphi)=\det (Z_{ij}) \sum_{m\in B_{k}(\mathbb{Z})}\delta_{m},  $$ where $\delta_{m}$ is the Dirac measure  at $m\in B_{k}(\mathbb{Z})$.
If $\chi_{k}: B\rightarrow B_{k}$ is induced by the dilation $y_{i} \mapsto k y_{i}$, $i=1, \cdots, n$, of $M_{\mathbb{R}}$, then $\frac{1}{k^{2}} \chi_{k}^{*}\varphi-  \overline{\varphi}$ is a well-defined function on $B$, and
$$\sup_{B} \left |\frac{1}{k^{2}} \chi_{k}^{*}\varphi- \overline{\varphi}\right | \rightarrow 0, $$
when $k\rightarrow\infty$.  Furthermore,   $$  \frac{1}{k^{n}} \chi_{k}^{*}{\rm MA}(\varphi) \rightharpoonup {\rm MA}(\bar{\varphi})=\det(Z_{ij})dy_{1}\wedge \cdots \wedge dy_{n}$$ in the weak sense.
  \end{theorem}

  In \cite{Liu}, the non-archimedean  Monge-Amp\`{e}re equation is  solved for degenerations of Abelian varieties, where the approximation of continuous  potential functions by piecewise linear functions  are also used. See \cite{BFJ} for the  non-archimedean  Monge-Amp\`{e}re equations for   more general cases. Here in Theorem \ref{thm02},  we are working  on the intersection complexes  instead of the dual intersection complexes as in \cite{BFJ,Liu}, and our piecewise linear functions are from the balanced embedded degenerations of Abelian varieties.

  We end this section by giving  a remark of Calabi-Yau manifolds balanced embedded by theta functions.
 In  Section 3 of  \cite{DKLR}, it is shown  that the Calabi-Yau hypersurfaces $$X_{t}=\{ [z_{0}, \cdots, z_{4}]\in \mathbb{CP}^{n} |t( z_{0}^{n+1}+ \cdots + z_{n}^{n+1})+ z_{0} \cdots z_{n}=0\}  $$ are balanced embedded for any $t\in \mathbb{C}$.  The proof involves two finite group actions on $(X_{t}, \mathcal{O}_{\mathbb{CP}^{n}}(1)|_{X_{t}})$. The first one is  $Ab_{n+1}=\{(a_{0}, \cdots, a_{n})|  a_{i} \in \mathbb{Z}_{n+1},  a_{0}+ \cdots+ a_{n}=0\}/ \mathbb{Z}_{n+1}$,  which  acts on $X_{t}$ by $$(z_{0}, \cdots, z_{n}) \mapsto (\zeta^{a_{0}}z_{0}, \cdots, \zeta^{a_{n}}z_{n})$$  where $\zeta=\exp \frac{2\pi \sqrt{-1}} {n+1}$. The second one is the symmetric group $ \mathbb{S}_{n+1}$ on $n+1$ elements, which acts on $X_{t}$ by translating $z_{0}, \cdots, z_{n}$. The same argument as in the proof of Theorem \ref{thm01} also shows the result of $X_{t}$ being balanced.
  On the other hand,  $z_{0}, \cdots, z_{n}$ as sections of $ \mathcal{O}_{\mathbb{CP}^{n}}(1)|_{X_{t}}$ are theta functions constructed in the Gross-Siebert program for $|t|\ll 1$ (at least for $n=3$) by Example 6.3 in   \cite{GHKS}, and thus $X_{t}$ are balanced embedded by theta functions.

\section{Construction of theta functions}
We recall the construction of theta functions on degenerations of Abelian varieties, and we   follow the arguments in  Example 6.1 of  \cite{GHKS} and Section 2 of  \cite{GS} closely. See also \cite{ABC} for the elliptic curve case.

 Let $(\mathcal{P}, \varphi)$ be the same as in the above section, i.e.
      the graph of $\varphi$ is the lower  boundary  of the upper  convex hull
      $$\mathrm{conv}    \{ (m, \overline{\varphi}(m))| m\in M \}\subset M_{\mathbb{R}}\times \mathbb{R},$$
       and  $\mathcal{P}$ be the rational  polyhedral decomposition of $ M_{\mathbb{R}}$ induced by $\varphi$.  If we define
  $$\Delta_{\varphi}=\{(m,r)\in M_{\mathbb{R}}\times \mathbb{R}| r \geqslant \varphi (m)\}, $$  the standard construction for toric degenerations  gives  a toric  variety $X_{\Delta_{\varphi}}$ with a line bundle $L_{\Delta_{\varphi}}$.

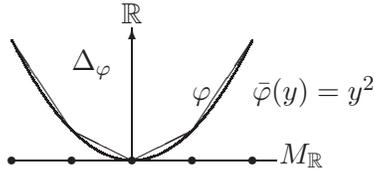
\setlength{\unitlength}{0.8cm}
\begin{center}
\begin{picture}(4.3,3.6)(-2.5,-0.25)
\put(-2,0){\line(1,0){4.4}}
\put(2.45,-.05){$M_{\mathbb{R}}$}
\put(0,0){\vector(0,1){2.2}}
\put(0,2.45){\makebox(0,0){$\mathbb{R}$}}
\qbezier(0.0,0.0)(1.0,0.0)
(2.0,2.0)
\qbezier(0.0,0.0)(-1.0,0.0)
(-2.0,2.0)
\put(0,0){\circle*{0.15}}\put(1,0){\circle*{0.15}}\put(-1,0){\circle*{0.15}}\put(2,0){\circle*{0.15}}\put(-2,0){\circle*{0.15}}
\put(0,0){\line(2,1){1}}\put(1,0.5){\line(2,3){1}}\put(0,0){\line(-2,1){1}}\put(-1,0.5){\line(-2,3){1}}
\put(2,1){$\bar{\varphi}(y)=y^{2}$} \put(1,1){$\varphi$} \put(-1,1.5){$\Delta_{\varphi}$}
\end{picture}
\end{center}

    For any $l\in\mathbb{N}$,
    $H^{0}(X_{\Delta_{\varphi}}, L_{\Delta_{\varphi}}^{l}) $
      is generated by {\em monomial} sections
 \begin{equation}\label{Z}
\{ \mathcal{Z}^{(m,r,l)}\mid \ \ \ \forall (m,r,l)\in C(\Delta_{\varphi}) \cap (M\times\mathbb{Z}\times \{l\})\}
 \end{equation}
where
       $$C(\Delta_{\varphi})=\overline{\{(lm',lr',l)|(m',r')\in \Delta_{\varphi}, l\in \mathbb{R}_{\geq 0}\}}\subset M_{\mathbb{R}}\times\mathbb{R}\times \mathbb{R}.   $$

 Note that we have a {\em canonical} regular function
 $$\bar{\pi}=\mathcal{Z}^{(0,1)}=\mathcal{Z}^{(0,1,0)}: X_{\Delta_{\varphi}}   \rightarrow  \mathbb{C}. $$  The toric boundary is   $\bar{X}_{0}=\{\mathcal{Z}^{(0,1)}=0\}$,  a toric variety with infinite many irreducible components, and $\bar{X}_{t}=\{\mathcal{Z}^{(0,1)}=t\}\cong T_{\mathbb{C}}$,
  for any $t
  \neq 0$.   We have a family of toric varieties $\bar{X}_{t}$ degenerating  to a singular toric varieties $\bar{X}_{0}$.

  The  degeneration of principally polarized   Abelian varieties  $(\pi: \mathcal{X}\rightarrow \Delta, \mathcal{L})$ is constructed as the quotient of an $M_{k}$-action on $(X_{\Delta_{\varphi}}, L_{\Delta_{\varphi}})$ as the following.

  \begin{lemma}\label{le00}   There is an $M$-action on $(X_{\Delta_{\varphi}}, L_{\Delta_{\varphi}})$ such that the projection  $\bar{\pi}$ is $M$-invariant, i.e.  $\bar\pi(m\cdot)=\bar\pi(\cdot)$ for any $m\in M$; the induced $M$-action on monomial rational functions is given by
    $$ \mathcal{Z}^{  (m,r)} \mapsto   \mathcal{Z}^{(m,   r+ Z(\gamma, m))},  \  \  \  \gamma\in M, \ \ \   (m,r)\in M\times\mathbb{Z};$$
    and the induced $M$-action on $H^{0}(X_{\Delta_{\varphi}}, L_{\Delta_{\varphi}})$ is given by  $$ \mathcal{Z}^{  (m,r,1)} \mapsto   \mathcal{Z}^{(m+ \gamma,   r+ \alpha_{\gamma}(m),1)},  \  \  \  \gamma\in M, \ \ \   (m,r)\in \Delta_{\varphi}$$  for monomial sections.
      \end{lemma}

   \begin{proof}  If  $\Sigma \subset N_{\mathbb{R}}\times \mathbb{R}$ denotes  the normal fan of $\Delta_{\varphi}$,  then one-dimensional rays of $\Sigma$ are   one-to-one correspondence  to the maximal dimensional cells in $\mathcal{P}$, and the primitive generator of a ray has the form $(-d\varphi |_{\sigma}, 1)$   for an $n$-dimensional cell $\sigma$ of $\mathcal{P}$.  Then  $M$ acts on $N\times \mathbb{Z} $ by $\check{T}^{0}_{\gamma}: (\mu,l)\mapsto (\mu- ld\alpha_{\gamma}, l) $  for any $\gamma \in M $, which preserves $\Sigma$, and thus  induces an $M $-action  on $X_{\Delta_{\varphi}}$.
      The dual $M $-action  on $M\times\mathbb{Z}  $ is the transpose, i.e. $T^{0}_{\gamma}:  (m,r) \mapsto (m,
  r+ d\alpha_{\gamma}(m))$  for any $\gamma \in M $. Thus the $M $-action  preserves the regular function  $\bar{\pi}=\mathcal{Z}^{(0,1)}$, and the induced $M$-action on monomial rational functions is given by
    $$ \mathcal{Z}^{  (m,r)} \mapsto   \mathcal{Z}^{T^{0}_{\gamma}  (m,r)}= \mathcal{Z}^{(m,   r+ Z(\gamma, m))},$$ for any $  \gamma\in M$ and $  (m,r)\in M\times\mathbb{Z}$.

  For constructing the $M$-action on  $L_{\Delta_{\varphi}}$, we consider
  the  $M $ action on  $C(\Delta_{\varphi}) $ defined  by
    \begin{equation}\label{e1.1} T_{\gamma}:  (m,r,l) \mapsto (m+ l\gamma,   r+ d\alpha_{\gamma}(m)+lc_{\gamma},l)  \  \  \  \  \  \gamma \in M,   \end{equation}  where
     $ d\alpha_{\gamma}(m)= Z(\gamma, m)$ by (\ref{e0.0}) and $c_{\gamma}=
     \frac{1}{2} Z(\gamma, \gamma)$.
         We have $ T_{\gamma}   (m,\varphi(m),1) =  (m+\gamma,   \varphi(m+\gamma),1)$ by (\ref{e0.2}).  This action lifts the $M$-action  on $X_{\Delta_{\varphi}}$ to an $M$-action  on $L_{\Delta_{\varphi}}^{l}$.
     More precisely  the $M$-action  on $L_{\Delta_{\varphi}}^{l}$ is given by  \begin{equation}\label{e1.1+} \mathcal{Z}^{  (m,r,l)} \mapsto \mathcal{Z}^{T_{\gamma}  (m,r,l)}=  \mathcal{Z}^{ (m+ l\gamma,   r+ d\alpha_{\gamma}(m)+lc_{\gamma},l)} \end{equation}  for monomial sections.
    \end{proof}

    For any $k\in \mathbb{N}$, $M_{k}=kM$ is a subgroup, and acts on $(X_{\Delta_{\varphi}}, L_{\Delta_{\varphi}})$ induced by the $M$-action in the above lemma.
Note that the map $\bar \pi$ is  $M_{k}$-invariant, and $M_k$ acts properly and discontinuously   on $\bar{\pi}^{-1}(\Delta)$   for the unit disc $\Delta \subset \mathbb{C}$.
    The quotient is the  degeneration of principally polarized   Abelian varieties $\pi: \mathcal{X}=\bar{\pi}^{-1}(\Delta)/M_{k}\rightarrow \Delta$, and $\pi^{-1}(t)=X_{t}=\bar{X}_{t}/M_{k}$. The central fiber $X_{0}$ of $\mathcal{X}$  is a union of  finite irreducible toric varieties, and the corresponding   intersection complex is $B_{k}=M_{\mathbb{R}}/M_{k} $  with rational  polyhedron decomposition $\tilde{\mathcal{P}}$ induced by $\mathcal{P}$. There is a one to one corresponding between the  $n$-dimensional  cells of $\tilde{\mathcal{P}}$ and the irreducible components of $X_{0}$. More precisely,
    for any $n$-dimensional  cell $\sigma$ of $\tilde{\mathcal{P}}$, we regard it as a rational polytope in $M_{\mathbb{R}} $, and it defines a polarized toric variety $(X_{\sigma}, L_{\sigma})$. The irreducible  component of  $X_{0}$ corresponding to $\sigma$ is isomorphic to  $X_{\sigma}$.
     The  quotient of   $L_{\Delta_{\varphi}}$ by the  $M_{k}$-action   is the relative ample line bundle $\mathcal{L}$ on $\mathcal{X}$.
      The restriction  of $\mathcal{L}$ on any irreducible  component $X_{\sigma}$ of $X_{0}$ is $L_{\sigma}$.
     The $M_{k}$-invariant sections descent  to  sections of  $\mathcal{L}$.

We claim that  and     $\mathcal{X}_{\Delta^{\circ}}=\pi^{-1}(\Delta^{\circ})$          is a base change of $\mathcal{X}_{\eta}$ (c.f. Section \ref{ppav})   via $t \mapsto t^{k}=s$.
       Note that for any $t\neq 0$,   $M_{k}$ acts on $\bar{X}_{t}=T_{\mathbb{C}}$ by
\begin{equation} \label{TZ}
 \mathcal{Z}^{(m,r)} :=\mathcal{Z} ^{m}t^{r}\mapsto \mathcal{Z}^{T^{0}_{\gamma} (m,r)}:=\mathcal{Z}^{(m,r)}t^{ Z(\gamma,m)}.
 \end{equation}
  Since  $t^{ k Z(k^{-1}\gamma,m)}  $,   $ k^{-1}\gamma \in M$,  we obtain that  $\bar{X}_{t}/M_{k} =X_{s}$  by (\ref{e0.100}) where  $ t^{k}=s$.

       For any $m \in B_{k}(\mathbb{Z})= M/M_{k}$, we define the theta function \begin{equation}\label{e1.2} \vartheta_{m}= \sum_{\gamma \in M_{k}} \mathcal{Z}^{T_{\gamma}  (m,\varphi(m),1)}=   \sum_{\gamma \in M_{k}} \mathcal{Z}^{(m+ \gamma, \varphi(m+\gamma),1)}, \end{equation} which is a section of $\mathcal{L}$ (cf.  Example 6.1 of  \cite{GHKS} and Section 2 of  \cite{GS}).   By abusing of notation,   we will use $m$ to denote {\em both} a point in $M$ and its image under
        the quotient map
$$M \longrightarrow B_{k}(\mathbb{Z})=M/kM\subset B_k=M_\RR/kM $$ without any confusing.
            We obtain a basis $\{\vartheta_{m}|m \in B_{k}(\mathbb{Z}) \}$ of $H^{0}(\mathcal{X}, \mathcal{L})$.  For any irreducible component $X_{\sigma}\subset X_{0}$,  $\vartheta_{m}|_{X_{\sigma}}$  is a monomial section of $L_{\sigma}$, and it is not a zero section if and only if $m\in \sigma$.
For any $t\neq0$,
  \begin{equation}\begin{split} \vartheta_{m}(w )& =  \sum_{\gamma \in M_{k}} \mathcal{Z}^{m+ \gamma}t^{ \varphi(m+\gamma)} \\  & =  \sum_{\gamma \in M_{k}} \exp \pi\sqrt{-1}(2\langle w, m+ \gamma\rangle+ \frac{\log t}{2 \pi \sqrt{-1}}Z(m+\gamma, m+\gamma)),   \end{split} \end{equation} by $\varphi(m+\gamma)=\bar{\varphi}(m+\gamma)$,   where $w=w_{1}e^{*}_{1}+\cdots +w_{n}e^{*}_{n}$.  Thus it is the classical theta function

 $$\vartheta_{m}(w)=\vartheta  \begin{bmatrix}
       m            \\[0.2em]
       0
     \end{bmatrix} (w,  \frac{\log t}{2 \pi \sqrt{-1}}Z_{ij} )
$$ on $X_{t}$.
If we regard $\vartheta_{m}|_{X_{t}}$ as on $X_{s}$, $s=t^{k}$,   then
$$ \vartheta_{m}(w) =   \sum_{\gamma' \in M} \exp \pi\sqrt{-1}(2k \langle w, \frac{m}{k}+ \gamma'\rangle+  k\frac{\log s}{2 \pi \sqrt{-1}}Z(\frac{m}{k}+\gamma',\frac{m}{k}+ \gamma')),$$
and a direct calculation shows
$$ \vartheta_{ m}(w+\mu + \frac{\log s}{2 \pi \sqrt{-1}}Z(p,\cdot))= \vartheta_{ m}(w)\exp k\pi \sqrt{-1}(-2\langle w,p\rangle- \frac{\log s}{2 \pi \sqrt{-1}}Z(p,p)),$$
for any $\mu\in N$ and $p\in M$. Thus $\sL|_{X_{t}} \cong \sL_{\eta}^{k}|_{X_{s}}$, i.e.
 $\sL|_{\mathcal{X}_{\Delta^{*}}}$ is the pull-back of  $\sL_{\eta}^{k}$.

\begin{remark}
Notice that, in  particular, the monodromy action $\log t\to \log t+2\pi \sqrt{-1}$ acts trivially on $\vth_m$ via
\begin{eqnarray*}
&&\vartheta  \begin{bmatrix}
       m            \\[0.2em]
       0
     \end{bmatrix} (w,  \left(\frac{\log t}{2 \pi \sqrt{-1}}+1\right )Z_{ij} ) \\
&=&   \sum_{\gamma \in M_{k}} \exp \pi\sqrt{-1}(2 \langle w, m+ \gamma \rangle+  (\frac{\log t}{2 \pi \sqrt{-1}}+1)Z(m+\gamma,m+ \gamma))\\
&=&\vartheta  \begin{bmatrix}
       m            \\[0.2em]
       0
     \end{bmatrix} (w,  \frac{\log t}{2 \pi \sqrt{-1}}Z_{ij} )
\end{eqnarray*}  since $ Z(m+\gamma,m+ \gamma)=2\bar{\varphi}(m+\gamma) \in 2\mathbb{Z} $  for $m\in B_k(\ZZ)$.
\end{remark}

 \begin{example}\label{extend}
  We illustrate the   explicit formula of the theta function $\vartheta_{m}$ in (\ref{e1.2}) in  local coordinates  for a special  1-dimensional family. Let   $M\cong \mathbb{Z}$, $k=3$,  and $\bar{\varphi}(y)=y^2$ on $M_{\mathbb{R}}$.
    Note that $(0, 0)$ is a vertex of $\Delta_{\varphi}$, and the $\mathbb{C}$-algebra $\mathbb{C} [T_{(0,0)}\Delta_{\varphi}\cap M]$  is generated by $z_{1}=\mathcal{Z}^{(1,1)}$, $z_{2}=\mathcal{Z}^{(-1,1)}$ and $t=\mathcal{Z}^{(0,1)}$,   where  $T_{(0,0)}\Delta_{\varphi}$ is  the    tangent cone of $\Delta_{\varphi}$ at $(0, 0)$.  The toric variety $Y_{0}={\rm Spec}(\mathbb{C} [T_{(0,0)}\Delta_{\varphi}\cap M])$ is defined in $\CC^3$  by  equation  $z_{1}z_{2}=t^{2}$, and an open subset of $Y_{0}$ is biholomorphic to a neighborhood $U_{0}$ of the zero strata of $X_{0}$ in $ \mathcal{X}$  corresponding to the vertex  $(0, 0)$.  Let us fix a trivialization of $\sL|_{U_0}\cong \sO_{Y_0}|_{U_0} $ via the indentification  $\mathcal{Z}^{(m,r,1)}\mapsto \mathcal{Z}^{(m,r)}$ for any $(m,r)\in \Delta_{\varphi}\cap M$.  Then  we have
      \begin{eqnarray*} & \vartheta_{0}& = 1+  \sum_{\nu \in  \mathbb{Z}, \nu>0} (z_{1}^{3\nu}+z_{2}^{3\nu})t^{9\nu^{2}-3\nu},\\ &\vartheta_{1}& = z_{1}+  \sum_{\nu \in  \mathbb{Z}, \nu>0}(z_{1}^{1+3\nu}t^{9\nu^{2}+3\nu}+z_{2}^{3\nu-1}t^{9\nu^{2}-9\nu+2}),  \\ & \vartheta_{2}& = z_{2}+  \sum_{\nu \in  \mathbb{Z}, \nu>0}(z_{1}^{3\nu -1}t^{9\nu^{2}-9\nu+2}+z_{2}^{3\nu+1}t^{9\nu^{2}+3\nu}),  \end{eqnarray*}      by (\ref{e1.2}).  In particular, $\vth_i$ extends to the central fiber.
 \end{example}

Notice that the  central extention of the product group $B_k(\ZZ)\times T_k$ with $T_k:=N/N_k$ is precisely the {\em finite Heisenberg group }
$\HH_k=\mu_k\times B_k(\ZZ)\times T_k$  (cf. \cite[Section 3]{Mum-th}) with the multiplication rule:
\begin{equation}
(\mu,a,b)\cdot (\mu',a',b')=(\mu\mu'\displaystyle\exp \frac{2\pi\sqrt{-1}\la b,a\ra}{k}, a+a',b+b')
\end{equation}
for any $(\mu,a,b),(\mu',a',b')\in \mu_k\times B_k(\ZZ)\times T_k$, where $\mu_k$ is the cyclotomic group of order $k$.

\begin{lemma}\label{le01}
The group
$B_k(\ZZ)\times T_k$ acts on $(\cX,\sL)$ and induces a representation  of $\HH_k$ on the $H^{0}(\mathcal{X}, \sL)=\mathrm{Span}_{\mathcal{H}}\{\vth_m\}_{m\in B_k(\ZZ)}$ via
\begin{itemize}
  \item[i)]   For any $a\in B_k(\ZZ)$
$$\fT_{a}\vartheta_{m} =  \vartheta_{a+m} $$

\item[ii)]  For any $b\in T_k$
$$ \fS_{b}\vartheta_{m}=\vartheta_{m} \exp \frac{2\pi\sqrt{-1}\langle b, m \rangle}{k},$$
  \end{itemize}
for all $m\in B_k(\ZZ)$, where $\mathcal{H}$ denotes the ring of holomorphic functions on $\Delta$.  In particular, the representation of $\HH_k$ on $H^0(\cX,\sL)$ is irreducible.
\end{lemma}

 \begin{proof}  The $M$-action on  $(X_{\Delta_{\varphi}}, L_{\Delta_{\varphi}})$ in Lemma \ref{le00} induces the $B_{k}(\mathbb{Z})$-action on $(\mathcal{X}, \mathcal{L})$, which acts on $H^{0}(\mathcal{X}, \mathcal{L})$ given by   \begin{eqnarray*}\fT_{a}\vartheta_{m}& = & \sum_{\gamma \in M_{k}} \mathcal{Z}^{T_{a}(m+ \gamma, \varphi(m+\gamma),1)} \\ & = & \sum_{\gamma \in M_{k}} \mathcal{Z}^{(m+ \gamma+a, \varphi(m+\gamma)+\alpha_{a}(m+\gamma) ,1)} \\ & = & \sum_{\gamma \in M_{k}} \mathcal{Z}^{(m+ \gamma+a, \varphi(m+\gamma+a),1)}\\ & = &\vartheta_{a+m}\end{eqnarray*} for any $m\in B_{k}(\mathbb{Z})$ and $a\in B_{k}(\mathbb{Z})$ by (\ref{e1.2}). We obtain i), and next we prove ii).

  Note that there is a natural injective     homeomorphism $\iota_{k}:T_{k} \hookrightarrow T_{\mathbb{C}} $ such that $$\mathcal{Z}^{m}(\iota_{k}(b))=\exp \frac{2\pi\sqrt{-1}\langle b, m\rangle}{k} $$ for any $b\in T_{k}$ and $m\in M$. The standard $T_{\mathbb{C}} $-action on $(X_{\Delta_{\varphi}}, L_{\Delta_{\varphi}}^{l})$ induces a $T_{k}$-action, which preserves the regular function $\bar{\pi}=\mathcal{Z}^{(0,1)} $, acts  on monomial  rational functions  by $$\mathcal{Z}^{(m, r)}\mapsto  \mathcal{Z}^{(m, r)}  \exp \frac{2\pi\sqrt{-1}\langle b, m\rangle}{k} $$ for any $(m,r)\in M\times\mathbb{Z}$,  and  acts on monomial sections  of $L_{\Delta_{\varphi}}^{l}$ via $$\mathcal{Z}^{(m, r,l)}\mapsto  \mathcal{Z}^{(m, r,l)}  \exp \frac{2\pi\sqrt{-1}\langle b, m\rangle}{k} $$ for any $(m,r,l)\in C( \Delta_{\varphi})$.  The induced $T_{k}$-action on $\bigoplus_{l=0}^{\infty}H^{0}(X_{\Delta_{\varphi}}, L_{\Delta_{\varphi}}^{l})$ commutes with the $M_{k}$-action by  $$ \mathcal{Z}^{T_{\gamma}(m, r,l)} \exp \frac{2\pi\sqrt{-1}\langle b, m\rangle}{k}= \mathcal{Z}^{(m+ l\gamma,   r+ d\alpha_{\gamma}(m)+lc_{\gamma},l)} \exp \frac{2\pi\sqrt{-1}\langle b, m+ l\gamma\rangle}{k}, $$ for any $\gamma\in M_{k}$ by (\ref{e1.1+}).   Thus
  the $T_{k}$-action commutes with the $M_{k}$-action on $(X_{\Delta_{\varphi}}, L_{\Delta_{\varphi}}^{l})$,  which induces a $T_{k}$-action on $(\pi:\mathcal{X}\rightarrow \Delta, \mathcal{L})$.    We denote $ \fS$  the induced $T_{k}$-action on  $H^{0}(\mathcal{X}, \mathcal{L})$, which satisfies $$
 \fS_{b}\vartheta_{m}= \sum_{\gamma \in M_{k}} \mathcal{Z}^{(m+ \gamma, \varphi(m+\gamma),1)}\exp \frac{2\pi\sqrt{-1}\langle b, m +\gamma \rangle}{k}=\vartheta_{m}\exp \frac{2\pi\sqrt{-1}\langle b, m \rangle}{k}, $$
for any $b\in T_{k}$ and any   $m\in B_{k}(\mathbb{Z})$,
by (\ref{e1.2}).
\end{proof}

\begin{remark}  The direct calculations show that on any $X_{t}$, $t\neq 0$,   the action of $B_k(\ZZ)\times T_k$ on $(X_t,\sL|_{X_t})$ is  given by $ w\mapsto w+\frac{b}{k}+ \frac{\log t}{2\pi\sqrt{-1}}Z(\cdot,a)$,
\begin{eqnarray*}(\fT_{a}\vartheta_{m})(w)& = &\exp\pi\sqrt{-1}(2\la w, a\ra+\frac{\log t}{2\pi\sqrt{-1}}Z(a,a))\vartheta_{m}(w+ \frac{\log t}{2\pi\sqrt{-1}}Z(\cdot,a))\\ & = & \vartheta_{a+m} (w),\\
( \fS_{b}\vartheta_{m})(w)&= & \vth_m(w+\frac{b}{k})=\exp \frac{2\pi\sqrt{-1}\langle b, m \rangle}{k}\vartheta_{m}(w),\end{eqnarray*} for any $m\in B_k(\ZZ)$ and $(a,b)\in B_k(\ZZ)\times T_{k}$.  In particular, there the finite torus $B_k(\ZZ)\times T_{k}$ acts on the image of the projective embedding $\Phi_k(X_t)$.
\end{remark}

\section{Proofs of Main Theorems}

Before we start the proof, let us recall the Hermitian metric on $ \sL|_{X_t}\to X_t$   with $s=t^k$  is given by
\begin{equation} h(w):=\exp\left ( \frac{2\pi}{\log|t|}Z^{ij}y_iy_j\right )= \exp k\left ( \frac{2\pi}{\log|s|}Z^{ij}y_iy_j\right )
\end{equation}
with
$$w_{i}=x_{i}+(\arg (s)-\sqrt{-1}\frac{\log |s|}{2\pi }  )\sum_{j=1}^nZ_{ij}y_{j},\  i=1, \cdots, n\ . $$  On $X_{t}\cong X_{s}$, we have $$\omega_{t}=-\sqrt{-1}\partial\overline{\partial}\log h=k \omega_{s},  $$ by (\ref{e0.10}).
\begin{lemma} $\forall a\in M, b\in N$ and function $f: N_{\mathbb{C}} \rightarrow \mathbb{C}$   we have
\begin{equation}
h(w)|\fS_b f|^2 =(h|f|^2)(w+\frac{b}{ k});
\text{  } h(w)|\fT_a f|^2 =(h|f|^2)\left (w+\frac{\log  s}{2\pi\sqrt{-1}}Z(\cdot, \frac{a}{ k})\right)\ .
\end{equation}
\end{lemma}

\begin{proof}   See   \cite[Proposition 3]{WY}.
\end{proof}
As a consequence, for  $a\in B_k(\ZZ),\ b\in T_k$ and $ f\in H^0(X_t,\sL|_{X_t})$   we have
\begin{eqnarray*}
\|(\fS_b f)(w)\|^2_\FS
&=&\int\frac{|(\fS_bf)(w)|^2}{\sum_{m}|\vth_m(w)|^2}\Phi_k^\ast\omega^n_\FS\\
&=&\int\frac{h|(\fS_bf)(w)|^2}{\sum_{m}h|\fS_b\vth_m(w)|^2}\Phi_k^\ast\omega^n_\FS\\
&=&\int\frac{(h|f|^2)\left(w+\frac{b}{k}\right)}{(\sum_{m}h|\vth_m|^2)\left(w+\frac{b}{k}\right)}(\fS_b\circ\Phi_k)^\ast\omega^n_\FS\\
&=&\int\left(\frac{|f|^2}{\sum_{m}|\vth_m|^2}\Phi_k^\ast\omega^n_\FS\right)\left(w+\frac{b}{k}\right)\\
&=&\|f(w)\|^2_\FS
\end{eqnarray*}
and
\begin{eqnarray*}
\|(\fT_a f)(w)\|^2_\FS
&=&\int\frac{|(\fT_af)(w)|^2}{\sum_{m}|\vth_m(w)|^2}\Phi_k^\ast\omega^n_\FS\\
&=&\int\frac{h|(\fT_af)(w+ \frac{\log  s}{2\pi\sqrt{-1}}Z(\cdot ,\frac{a}{ k}))|^2}{\sum_{m}h|\fT_a\vth_m(w+ \frac{\log  s}{2\pi\sqrt{-1}}Z(\cdot ,\frac{a}{ k}))|^2}(\fT_a\circ\Phi_k)^\ast\omega^n_\FS\\
&=&\int\left(\frac{|f|^2}{\sum_{m}|\vth_m|^2}\Phi_k^\ast\omega^n_\FS\right)(w+ \frac{\log  s}{2\pi\sqrt{-1}}Z(\cdot ,\frac{a}{ k}))\\
&=&\|f(w)\|^2_\FS\ .
\end{eqnarray*}
Hence the  finite group generated by image of
$$ \{\fT_a,\fS_b \mid (a,b)\in B_k(\ZZ)\times T_k\} \subset \mathrm{GL}(H^0(X_{ s},\sL|_{X_{ s}}))$$
actually lies in $\mathrm{U}(k^n)$ with respect to  the Fubini-Study metric induced via the embedding
 $$
 \Phi_k=[\vth_{m_1},\cdots,\vth_{m_{k^n}}]:\cX\longrightarrow \CC\PP^{k^n-1}\ .
 $$
 \begin{proof}[Proof of Theorem \ref{thm01}]
 It follows from above that the action generated by $B_k(\ZZ)\times T_k$ via   $\fT$ and $\fS$ lies in $\mathrm{U}(k^n)$. And   Lemma \ref{le01}  implies that sections $\{\vth_m\}_{m\in B_k(\ZZ)}$ forms an orthonormal basis  with respect to the pull back of Fubini-Study metric via the map $\Phi_k$, that is, the embedding $\Phi_k$ is balanced for each $t\in \Delta^\circ$.

 On the other hand, $\Phi_k(X_t)$  being balanced for each $t\ne 0$ implies that the  Chow point for $\Phi_k(X_t)$ lies on the $0$-level set the moment map

 \begin{equation*}
\begin{array}{cccc}
  \mu_{{\rm SU}}:{\rm Chow}_{\mathbb{CP}^{k^n-1}}(d,n) & \overset{}{ \xrightarrow{\hspace*{1.7cm}}} & \mathfrak{su}(k^n) \\
X_t & \longmapsto & \sqrt{-1}\cdot \displaystyle \int_{X_t}\left(\frac{\vth_m \bar\vth_{m'}}{\sum_m|\vth_m|^2}-\frac{\delta_{mm'}}{k^n}\right )\frac{ \Phi_k^\ast\omega^n_\FS}{n!} &
\end{array}%
\end{equation*}%
  of the $\mathrm{SU}(k^n)$-action on the Chow variety of $n$-dimensional degree $d$ cycles in $\CC\PP^{k^n-1}$ (c.f. \cite[Proposition 17]{W} and \cite[Theorem 1.4]{ZS}),
  which is proper via standard Kirwan-Kempf-Ness theory  in \cite{Kir} (c.f. also in \cite{Th}).  Notice that $\{\vth_m(\cdot,   \frac{\log t}{2\pi\sqrt{-1}}Z)\}_m$  vary  holomorphically with respect to $t$ and can be extended to $X_0$  by the construction of theta functions in Section 3, these imply that $\Phi_k$ has bounded image in $\CC\PP^{k^n-1}$, by Riemann mapping Theorem, the {\em unique} continuous extension $\Phi_k(X_0)$ must lies in $ \mu_{\rm SU}^{-1}(0)$, that is, the embedding $\Phi_k(X_0)$ is balanced as well.

  \end{proof}
   \begin{proof}[Proof of Theorem \ref{thm02}] For any $m\in M$, we denote $\check{m}\subset N_{\mathbb{R}}$ the dual polytope of $m$ with  respect   to $\varphi$. More precisely, if $ \check{m}^{o}\subset N_{\mathbb{R}}\times \mathbb{R}$ denotes the dual cone of the tangent cone  $T_{(m,\varphi(m))}\Delta_{\varphi} \subset M_{\mathbb{R}}\times \mathbb{R}$ of $\Delta_{\varphi}$ at the vertex $(m, \varphi(m))$, then $\check{m}= \check{m}^{o} \cap (N_{\mathbb{R}}\times \{1\})$.   The Monge-Amp\`{e}re measure of $\varphi$ is  $${\rm MA}(\varphi)=\sum_{m\in B_{k}(\mathbb{Z})}{\rm Vol}(\check{m})\delta_{m},  $$ where ${\rm Vol}(\check{m})$ is the Euclidean volume of $\check{m}$ (cf. Proposition 2.7.4 in  \cite{BPS}).
     Since the   $M$-action $\check{T}^{0}$ on $N_{\mathbb{R}}\times \mathbb{R} $  preserves the fan  $\Sigma$ of  $\Delta_{\varphi}$, we obtain $\check{T}^{0}_{m-m'}(\check{m}')=\check{m}$ for any two $m$ and $m'\in M$, and thus ${\rm Vol}(\check{m})=V$ is independent of $m$.

     Let $D$ be the fundamental domain of the    $M_{k}$-action $\check{T}^{0}$ on $N_{\mathbb{R}}\times \{1\} $. Since $\check{T}^{0}_{\gamma} (0,l)= (- d\alpha_{\gamma}, l) =(-Z(\gamma,\cdot),1)$  for any $\gamma \in M $, we let  $D$ be the convex hull of $(0,1), (-Z(k e_{1},\cdot),1), \cdots, (-Z(k e_{n},\cdot),1)$  in $N_{\mathbb{R}}\times \{1\} $, where $e_{1}, \cdots, e_{n}$ is a basis of $M$. The Euclidean volume of $D$ is ${\rm Vol}(D)=k^{n}\det (Z_{ij})$, where $Z_{ij}=Z(e_{i},e_{j}) $, and we have $$  {\rm Vol}(D)=\sum_{m\in B_{k}(\mathbb{Z})}{\rm Vol}(\check{m})=k^{n} V.$$ We obtain $V=\det (Z_{ij})$, and the conclusion $${\rm MA}(\varphi)=\det (Z_{ij}) \sum_{m\in B_{k}(\mathbb{Z})}\delta_{m}. $$

      Let $\chi_{k}: B\rightarrow B_{k}$ be the  diffeomorphism  induced by the dilation $y_{i} \mapsto k y_{i}$, $i=1, \cdots, n$, of $M_{\mathbb{R}}$. For any smooth function $f$ on $B$,  $$  \frac{1}{k^{n}}\int_{B} f \chi_{k}^{*}(  \sum_{m\in B_{k}(\mathbb{Z})}\delta_{m})=\frac{1}{k^{n}}   \sum_{m'\in (\frac{1}{k}M)/M}f(m')  \rightarrow \int_{B}f dy_{1}\wedge \cdots \wedge dy_{n}$$ when $k\rightarrow \infty$, which implies that $$  \frac{1}{k^{n}} \chi_{k}^{*}{\rm MA}(\varphi) \rightharpoonup \det(Z_{ij})dy_{1}\wedge \cdots \wedge dy_{n}={\rm MA}(\bar{\varphi})$$ in the weak sense.

      By (\ref{e0.1}) and (\ref{e0.2}),   $(\varphi-\bar{\varphi})(y+\gamma)=(\varphi-\bar{\varphi})(y)$ for any $y\in M_{\mathbb{R}}$ and any $\gamma \in M$, and  $\varphi-\bar{\varphi}$ is a  well-defined function on $B_{k}$.   Since $$ \sup_{ M_{\mathbb{R}}}|\varphi-\bar{\varphi}|=\sup_{y \in  M_{\mathbb{R}}/ M}|\varphi-\bar{\varphi}|(y),$$ we obtain
$$
\sup_{ B}\left | \frac{1}{k^{2}} \chi_{k}^{*} \varphi-  \frac{1}{k^{2}} \chi_{k}^{*} \bar{\varphi}\right| =\frac{1}{k^{2}}\sup_{ M_{\mathbb{R}}}|\varphi-\bar{\varphi}|   \rightarrow 0,
$$
when $k\rightarrow\infty$, and the conclusion by $\frac{1}{k^{2}}\bar{\varphi}(ky)=\bar{\varphi}(y)$.
  \end{proof}
\section{Appendix}
In this section, we state the following Theorem which unified the proof of balanced embedding for projective space and princpally polarized Abelian varieties.

\newcommand{\SU}{\mathrm{SU}}
\begin{theorem}Let $X\subset \PP^N$ be a subvariety and $G<\SU(N+1)$ be a compact subgroup which leaves the embedding $X\hookrightarrow\PP^N$ invariant.  Suppose the centralizer of $c_G< \SU(N+1)$ of $G$ inside $\SU(N+1)$ is trivial. Then the embedding $X\subset \PP^N$ is balanced, i.e.
\begin{equation}\label{mu}
\int_X\mu_{\PP^N}\frac{\omega_\FS^n}{n!}=0.
\end{equation}
\end{theorem}
\begin{proof}
Notice that  the moment map of $\SU(N+1)$-action on the Chow variety $\mathrm{Chow}_{\PP^N}(n,d)$ of dimension $n$ and degree $d$ cycle in $\PP^N$ is precisely given by \eqref{mu}.  In particular, it is $\SU(N+1)$, and hence $G$-equivariant.  This implies that for any $g\in G$
\begin{eqnarray*}\label{mu}
&&\int_{ X}\mu_{\PP^N}\frac{\omega_\FS^n}{n!}
=\int_{g\cdot X}\mu_{\PP^N}\frac{\omega_\FS^n}{n!}=\int_{X}\mu_{\PP^N}\circ g\frac{\omega_\FS^n}{n!}\\
&=&\mathrm{Ad}_g\left(\int_{X}\mu_{\PP^N}\frac{\omega_\FS^n}{n!}\right).
\end{eqnarray*}
By our assumption, we have $\displaystyle \int_{ X}\mu_{\PP^N}\frac{\omega_\FS^n}{n!}\in \mathfrak{c}_G=0$, where $\mathfrak{c}_G=\mathrm{Lie}(c_G)$ is the Lie algebra. And our proof is thus completed.
\end{proof}


\begin{thebibliography}{99}
\bibitem{ABC} P. Aspinwall,  T.  Bridgeland, A.  Craw, M.  Douglas, M.  Gross, A. Kapustin, G.  Moore, G. Segal, B. Szendr\H{o}i, P. Wilson, {\em
Dirichlet branes and mirror symmetry.}
Clay Mathematics Monographs, 4. American Mathematical Society, (2009).
 \bibitem{AN}  V. Alexeev, I. Nakamura,  {\em On Mumford's construction of degenerating Abelian varieties},  \rm Tohoku Math. J. 51 (1999), 399--420.
  \bibitem{BL}   C.  Birkenhake, H. Lange, {\em
Complex abelian varieties,}   Grundlehren der Mathematischen Wissenschaften, 302. Springer-Verlag, Berlin, 2004.
  \bibitem{BFJ} S. Boucksom, C.  Favre, M.  Jonsson, {\em Solution to a non-Archimedean Monge-Amp\`{e}re equation,}  J. Amer. Math. Soc. 28  no. 3, (2015),  617--667.
  \bibitem{BPS}  J.  Burgos Gil, P.  Philippon, M.  Sombra, {\em  Arithmetic geometry of toric varieties. Metrics, measures and heights.}  Ast\'{e}risque No. 360 (2014).
 \bibitem{Dos} S.   Donaldson, {\em  Scalar curvature and projective embeddings I,}  J. Differential Geom. 59 no. 3, (2001),  479--522.
   \bibitem{DKLR}  M.  Douglas, R.  Karp, S.  Lukic, R.  Reinbacher,  {\em  Numerical Calabi-Yau metrics},  J. Math. Phys. 49 no. 3, (2008),  032302, 19 pp.
  \bibitem{Gro}  M.  Gross, {\em  Mirror symmetry and the Strominger-Yau-Zaslow conjecture.}   Current developments in mathematics 2012,  Int. Press, Somerville, MA, (2013),  133--191.
 \bibitem{GHKS}  M. Gross,  P.  Hacking, S.  Keel, B.  Siebert,  {\em   Theta functions on varieties with effective anti-canonical class}, \rm  arXiv:1601.07081.
  \bibitem{GHKS2}  M. Gross,  P.  Hacking, S.  Keel, B.  Siebert,  {\em   Theta functions for K3 surface}, preprint.
  \bibitem{GS1}  M. Gross, B.  Siebert, {\em  From real affine geometry to complex geometry.}  Ann. of Math. (2) 174 no.3, (2011),  1301--1428.
  \bibitem{GS}   M.  Gross, B.  Siebert,  {\em Theta functions and mirror symmetry},   arXiv:1204.1991.
   \bibitem{GW}  M. Gross,  P.M.H. Wilson,  {\em    Large complex structure limits of
K3 surfaces}, \rm J. Diff. Geom.  55 (2000),  475--546.
 \bibitem{GTZ} M. Gross, V. Tosatti, Y. Zhang, {\em  Collapsing of Abelian Fibred Calabi-Yau Manifolds},  Duke Math. J.  162 (2013),  517--551.
 \bibitem{GTZ2} M. Gross, V. Tosatti, Y. Zhang,  {\em Gromov-Hausdorff collapsing of Calabi-Yau manifolds},  arXiv:1304.1820,  to appear in  Comm. Anal. Geom.
\bibitem{Kir}F. Kirwan {\em Cohomology of Quotients in Symplectic and Algebraic Geometry } Mathematical Notes, 31. Princeton University Press, Princeton, NJ, 1984.
\bibitem{KS} M. Kontsevich, Y. Soibelman, {\em  Homological mirror symmetry and torus fibrations
 }, in \it Symplectic geometry and mirror symmetry, \rm  World Sci.
 Publishing, (2001), 203--263.
  \bibitem{KS2} M. Kontsevich, Y. Soibelman, {\em  Affine Structures and Non-Archimedean Analytic Spaces}, in \it The Unity of Mathematics, \rm  Progress in Mathematics Volume 244, Springer,  (2006),  321--385.
   \bibitem{Liu}    Y.  Liu, {\em A non-Archimedean analogue of the Calabi-Yau theorem for totally degenerate abelian varieties.}  J. Differential Geom. 89 no.1,  (2011),  87--110.
 \bibitem{Mum}  D. Mumford,  {\em An analytic construction of degenerating Abelian varieties over complete rings,} Compositio Math. 24 No.3, (1972), 239--272.
  \bibitem{Mum-th}  D. Mumford,  {\em Tata lectures on Theta. I. With the assistance of C. Musili, M. Nori, E. Previato and M. Stillman.} Progress in Mathematics, {\bf 28}. Birkh\"auser Boston, Inc., Boston, MA, 1983. xiii+235 pp.
 \bibitem{No}  Y.   Nohara, {\em  Projective embeddings and Lagrangian fibrations of abelian varieties,}  Math. Ann. 333 no. 4, (2005),  741--757.
  \bibitem{TW}  N.  Trudinger, X.  Wang, {\em  The Monge-Amp\`{e}re equation and its geometric applications,}  Handbook of geometric analysis.  Adv. Lect. Math. (ALM), 7,  No. 1, (2008),  467--524.
       \bibitem{RS} H. Ruddat, B. Siebert, {\rm Canonical Coordinates in Toric Degenerations},  arXiv:1409.4750.
       \bibitem{Th}  R. P.  Thomas, {\em Notes on GIT and symplectic reduction for bundles and varieties},   Surv. Differ. Geom., 10, Int. Press, (2006),  221--273.
       \bibitem{SYZ}  A. Strominger, S.T. Yau, E. Zaslow,
 {\em  Mirror symmetry is T-duality,}  \rm Nuclear Physics B, Vol. 479, (1996), 243--259.
\bibitem{W} X. Wang, {\em
Moment map, Futaki invariant and stability of projective manifolds.}
Comm. Anal. Geom. 12 no.5,  (2004),  1009--1037. 
 \bibitem{WY} X.  Wang, H. Yu,  {\em Theta function and Bergman metric on abelian varieties}, \rm  New York J. Math. 15 (2009), 19--35.
  \bibitem{Yau1}  S.T. Yau,  {\em On the Ricci curvature of a compact  K\"{a}hler
manifold and complex Monge-Amp\`{e}re  equation I}, \rm  Comm. Pure
Appl. Math.  31  (1978),  339--411.
 \bibitem{ZS} S. Zhang, {\em Heights and reductions of semi-stable varieties,} Compositio Math. 104 (1996), 77--105.
\end{thebibliography}
\end{document}